\documentclass[11pt, makeidx]{amsart}

\usepackage{enumerate}
\usepackage{amssymb,amsfonts,amsthm}
\usepackage{srcltx}
\usepackage{graphicx}
\usepackage{hyperref}

\usepackage{float}
\usepackage[OT2,OT1]{fontenc}
\newcommand\cyr{%
\renewcommand\rmdefault{wncyr}%
\renewcommand\sfdefault{wncyss}%
\renewcommand\encodingdefault{OT2}%
\normalfont
\selectfont}
\DeclareTextFontCommand{\textcyr}{\cyr}
\usepackage{color}
\makeindex

\makeatletter
\def\sssub{\@startsection{paragraph}{4}}

\renewcommand\paragraph{\@startsection{paragraph}{4}{\z@}{1.25ex}{0.0001pt}{\normalfont\normalsize\em}}
\makeatother

\numberwithin{paragraph}{subsubsection}

\newcommand{\ee}{\end{enumerate}}
\newcommand{\beq}{\begin{equation}}
\newcommand{\eeq}{\end{equation}}

\newcommand{\F}{{\mathbb{F}}}
\newcommand{\Z}{{\mathbb{Z}}}

\newcommand{\Stab}{\ensuremath{\mathrm{Stab}}}

\newcommand{\K}{\mathcal K}

\def\CH{\mathcal H}

\newtheorem{theorem}{Theorem}[section]
\newtheorem{theoremIntro}{Theorem}

\newtheorem{proposition}[theorem]{Proposition}

\newtheorem{lemma}[theorem]{Lemma}
\newtheorem{definition}[theorem]{Definition}

\theoremstyle{remark}

{}

\newcommand{\bea}{\begin{eqnarray*}}
\newcommand{\eea}{\end{eqnarray*}}

\begin{document}

\title{Fra\"{i}ss\'{e} limits of limit groups}

\author{Olga Kharlampovich,  Alexei Myasnikov, Rizos Sklinos}
%\thanks{$^1$ Partially supported by NSF grant DMS-0700811, $^2$ Partially supported by NSF
%grants DMS-0700811 and DMS-0914773. }

\address{Olga Kharlampovich, Department of Mathematics and Statistics, Hunter College, City University of New York, New York, NY, 10065, U.S.A.}
\email{okharlampovich@gmail.com}

\address{Alexei Myasnikov, Stevens Institute of Technology, Hoboken, NJ, 07030 U.S.A.}
\email{amiasnikov@gmail.com}

\address{Rizos Sklinos, Stevens Institute of Technology, Hoboken, NJ, 07030 U.S.A.}
\email{rizozs@gmail.com}

\begin{abstract} 
We modify the notion of a Fra\"{i}ss\'{e} class and show that various interesting classes of groups, notably the class of nonabelian limit groups and the class of 
finitely generated elementary free groups, admit Fra\"{i}ss\'{e} limits. 

Furthermore, we rediscover Lyndon's  $\Z[t]$-exponential completions of countable torsion-free CSA groups, as   
Fra\"{i}ss\'{e} limits with respect to extensions of centralizers. 

\vspace{.2cm}
Dedicated to the memory of Charles Sims.   
\end{abstract}

\maketitle

\tableofcontents

\section{Introduction}

Fra\"{i}ss\'{e} constructions have been introduced by Roland Fra\"{i}ss\'{e} in \cite{RF}, where he observed that one can see the class of finite linear orders as 
approximations of $(\mathbb{Q}, <)$ and extended this observation to more general structures satisfying certain properties 
(still in a finite relational language). He showed, in particular, how one can construct the ordering of the rational numbers as a direct limit of finite linear orders using amalgamations. 
Furthermore, his construction implies the countability, the universality and the homogeneity of the limit structure, as well 
as its uniqueness with respect to those properties. 
The idea of Fra\"{i}ss\'{e} has been proved extremely fruitful and it has been adapted and used to 
discover mathematical structures with certain universal and homogeneous properties. The applicability of Fra\"{i}ss\'{e}'s ideas in many different areas of mathematics 
demonstrate their power and usefulness. The {\em random graph}
in graph theory and {\em Philip Hall's universal locally finite group} in group theory are conspicuous examples amongst many.   
%a countable simple group that contains all finite groups and moreover any two isomorphic finite subgroups are conjugate in it. As a matter of fact there is 
%a unique countable group (up to isomorphism) with the properties described above. 

In this paper we study generalized Fra\"{i}ss\'{e} limits in the class  $\mathcal{F}$  of  all nonabelian  limit groups. The class $\mathcal{F}$, which coincides with the class of nonabelian finitely 
generated $\omega$-residually free or universally free groups, is nowadays recognized as the main study of the algebraic geometry or model theory of free  groups.  One cannot apply directly Fra\"{i}ss\'{e}'s methods to the class $\mathcal{F}$, 
put differently the class $\mathcal{F}$  is not a Fra\"{i}ss\'{e} class with respect to the standard embeddings of groups. However, it is more natural and more advantageous to consider limit groups as a category with $\forall$-embeddings 
(after all, limit groups are groups $\forall$-equivalent to a free non-abelian  group). In this case it is a "Fra\"{i}ss\'{e} class" (more precisely "Fra\"{i}ss\'{e}  category", 
but we stay with the old name here), so the Fra\"{i}ss\'{e} limit exists, it is unique up to an isomorphism, and have a nice universal property.  
Furthermore, going this way, one can naturally consider a subclass $\mathcal{F}_e \subset \mathcal{F}$ of all elementary free groups (those limit groups which are elementarily equivalent to a nonabelian free group) 
with respect to elementary embeddings. Again, this is a Fra\"{i}ss\'{e} class, so the limits exists, is unique, and homogeneous.  These uniquely defined  groups, say  $G$ and $G_e$ (the Fra\"{i}ss\'{e} limits 
of the classes $\mathcal{F}$ and $\mathcal{F}_e$) are very interesting objects in their own rights. Their algebraic structure is a mystery, which we did not attend to in this paper.

In more detail, the first main theorem of this paper is: 

\begin{theoremIntro}
The class $\mathcal{F}$  forms a $\forall$-Fra\"{i}ss\'{e} class. In particular there exists a countable group $G$ with the following properties:
\begin{itemize}
 \item the universal age of $G$, $\forall$-$age(G)$, is the class of nonabelian limit groups;
 \item for any two finitely generated isomorphic $\forall$-subgroups of $G$, there is an automorphism of $G$ extending the isomorphism;
  \item the group $G$ is the union of a $\forall$-chain of limit groups.
\end{itemize}
Moreover, any other countable group with the above properties is isomorphic to $G$.
\end{theoremIntro}  

For the class of finitely generated elementarily free groups  we prove the following: 

\begin{theoremIntro}
The class $\mathcal{F}_e$ of elementary free groups forms an $e$-Fra\"{i}ss\'{e} class. In particular there exists a countable group $G_e$ with the following properties:
\begin{itemize}
 \item the elementary age of $G_e$, $e$-$age(G_e)$, is the class of elementary free groups;
 \item the group $G_e$ is homogeneous, i.e. whenever two tuples from $G_e$ have the same type there is an automorphism taking one to the other;
  \item the group $G_e$ is the union of an elementary chain of elementary free groups. 
\end{itemize}
Moreover, any other countable group with the above properties is isomorphic to $G_e$.
\end{theoremIntro}

In addition, we prove that Lyndon's $\Z[t]$-exponential completion of a countable torsion-free CSA group is homogeneous with respect to a special class of subgroups, thus can be obtained as a 
Fra\"{i}ss\'{e} limit.

\section{Preliminaries}\label{Pre}

For the benefit of the reader we collect in this section a brief overview of the classical Fra\"{i}ss\'{e} theory as well as natural generalizations. The material in this 
section is well-known to model theorists (see \cite{DE} for example). Nevertheless, we record it.  
%The last subsection briefly discusses Bass-Serre theory and basic constructions of combinatorial group theory such as amalgamated free products and HNN-extensions.

For the rest of the section we fix a countable first-order language $\mathcal{L}$.

\subsection{Fra\"{i}ss\'{e} limits}

Let $\mathcal{M}$ be an $\mathcal{L}$-structure. The age of $\mathcal{M}$, $age(\mathcal{M})$, is the class of all finitely generated 
substructures of $\mathcal{M}$.

\begin{definition}
Let $\mathcal{K}$ be a countable (with respect to isomorphism types) non-empty class of finitely generated $\mathcal{L}$-structures with the following properties:
\begin{itemize}
 \item (IP) the class $\mathcal{K}$ is closed under isomorphisms;
 \item (HP) the class $\mathcal{K}$ is closed under finitely generated substructures;
 \item (JEP) if $\mathcal{A}_1$, $\mathcal{A}_2$ are in $\mathcal{K}$, then there is $\mathcal{B}$ in $\mathcal{K}$ 
 and embeddings $f_i:\mathcal{A}_i\rightarrow \mathcal{B}$ for $i\leq 2$;
 \item (AP) if $\mathcal{A}_0$, $\mathcal{A}_1$, $\mathcal{A}_2$ are in $\mathcal{K}$ and $f_i:\mathcal{A}_0\rightarrow \mathcal{A}_i$ for $i\leq 2$ are embeddings,   
 then there is $\mathcal{B}$ in $\mathcal{K}$ and embeddings $g_i:\mathcal{A}_i\rightarrow \mathcal{B}$ for $i\leq 2$ with $g_1\circ f_1=g_2\circ f_2$.  
\end{itemize}
Then $\mathcal{K}$ is a Fra\"{i}ss\'{e} class.
\end{definition}

The classical example that motivated the above definition is the class of finite linear orders in $\mathcal{L}:=\{<\}$. If one is concerned with groups, in the language of groups, 
then it is not hard to see that finitely generated groups or even finitely presented groups do not form a Fra\"{i}ss\'{e} class. On the other hand,  
the class of finitely generated abelian groups (or even finitely generated free abelian groups) is a Fra\"{i}ss\'{e} class. It is more challenging to show that the class of finite groups is a Fra\"{i}ss\'{e} class. 

Fra\"{i}ss\'{e} classes are important for the following reason.

\begin{theorem}[Fra\"{i}ss\'{e}'s theorem]
Let $\mathcal{K}$ be a Fra\"{i}ss\'{e} class. Then there exists a countable $\mathcal{L}$-structure $\mathcal{M}$ such that:
\begin{itemize}
 \item the age of $\mathcal{M}$ is exactly $\mathcal{K}$; 
 \item the $\mathcal{L}$-structure $\mathcal{M}$ is ultrahomogeneous, i.e. every isomorphism between finitely generated substructures of $\mathcal{M}$ extends to 
 an automorphism of $\mathcal{M}$.
\end{itemize}
Moreover, any other countable $\mathcal{L}$-structure with the above properties is isomorphic to $\mathcal{M}$.
\end{theorem}

\subsection{Universal Fra\"{i}ss\'{e} limits}

We generalize Fra\"{i}ss\'{e} constructions by strengthening the properties embeddings preserves. We will call an embedding that preserves universal (or equivalently $\Pi_1^{0}$) formulas a 
$\forall$-embedding. We will denote $\forall$-embeddings by $\rightarrow_{\forall}$. We remark that if $\mathcal{A}\subseteq_{\forall}\mathcal{B}$, i.e. the inclusion map is a $\forall$-embedding, 
then $\mathcal{A}$ is existentially closed in $\mathcal{B}$, i.e. for any quantifier free formula $\phi(\bar{x})$ with parameters in $\mathcal{A}$,  
if $\mathcal{B}\models\exists\bar{x} \phi(\bar{x})$, then  $\mathcal{A}\models\exists\bar{x} \phi(\bar{x})$. Observe that the other direction of the previous implication holds trivially by the 
fact that $\mathcal{A}$ is a substructure of $\mathcal{B}$. 

Let $\mathcal{M}$ be an $\mathcal{L}$-structure. The universal age of $\mathcal{M}$, $\forall$-$age(\mathcal{M})$, is 
the class of all finitely generated $\forall$-substructures of $\mathcal{M}$ (or substructures existentially closed in $\mathcal{M}$) up to isomorphism. 

\begin{definition}[Universal Fra\"{i}ss\'{e} class]
Let $\mathcal{K}$ be a countable (with respect to isomorphism types) non-empty class of finitely generated $\mathcal{L}$-structures with the following properties:
\begin{itemize}
 \item (IP) the class $\mathcal{K}$ is closed under isomorphisms;
 \item ($\forall$-HP) the class $\mathcal{K}$ is closed under finitely generated $\forall$-substructures;
 \item ($\forall$-JEP) if $\mathcal{A}_1$, $\mathcal{A}_2$ are in $\mathcal{K}$, then there is $\mathcal{B}$ in $\mathcal{K}$ 
 and $\forall$-embeddings $f_i:\mathcal{A}_i\rightarrow_{\forall} \mathcal{B}$ for $i\leq 2$;
 \item ($\forall$-AP) if $\mathcal{A}_0$, $\mathcal{A}_1$, $\mathcal{A}_2$ are in $\mathcal{K}$ and $f_i:\mathcal{A}_0\rightarrow_{\forall} \mathcal{A}_i$ for $i\leq 2$ are $\forall$-embeddings,   
 then there is $\mathcal{B}$ in $\mathcal{K}$ and $\forall$-embeddings $g_i:\mathcal{A}_i\rightarrow_{\forall} \mathcal{B}$ for $i\leq 2$ with $g_1\circ f_1=g_2\circ f_2$.  
\end{itemize}
Then $\mathcal{K}$ is a universal Fra\"{i}ss\'{e} class or for short a $\forall$-Fra\"{i}ss\'{e} class.
\end{definition}

In complete analogy to classical Fra\"{i}ss\'{e} limits one can prove:

\begin{theorem}\label{forallFraisse}
Let $\mathcal{K}$ be a $\forall$-Fra\"{i}ss\'{e} class. Then there exists a countable $\mathcal{L}$-structure $\mathcal{M}$ such that:
\begin{itemize}
 \item the $\forall$-age of $\mathcal{M}$ is exactly $\mathcal{K}$; 
 \item the $\mathcal{L}$-structure $\mathcal{M}$ is weakly $\forall$-homogeneous, i.e. every isomorphism between finitely generated $\forall$-substructures of $\mathcal{M}$ extends to 
 an automorphism of $\mathcal{M}$;
  \item the $\mathcal{L}$-structure $\mathcal{M}$ is the union of a $\forall$-chain of $\mathcal{L}$-structures in $\mathcal{K}$.
\end{itemize}
Moreover, any other countable $\mathcal{L}$-structure with the above properties is isomorphic to $\mathcal{M}$.
\end{theorem}

Observe that in this case in order to obtain uniqueness of the limit we need to additionally assume that $\mathcal{M}$ is a union of a $\forall$-chain. The main reason for that 
is that a substructure generated by a finite tuple is not necessarily a $\forall$-substructure, thus we cannot assume that a countable $\mathcal{L}$-structure 
is exhausted by its finitely generated $\forall$-substructures. 

As yet another difference from the classical Fra\"{i}ss\'{e} theorem is that the $\forall$-Fra\"{i}ss\'{e} limit is not $\forall$-homogeneous, but only weakly $\forall$-homogeneous. 
Recall, that an $\mathcal{L}$-structure $\mathcal{M}$ is $\forall$-homogeneous if whenever two finite tuples from $\mathcal{M}$ satisfy the same universal formulas, there is 
an automorphism of $\mathcal{M}$ taking one to the other. This can be ``corrected'' if one considers partial $\forall$-embeddings. If $A_0$ is a subset of the domain of the 
$\mathcal{L}$-structure $\mathcal{A}$, then the map $f:A_0\rightarrow \mathcal{B}$ is a partial $\forall$-embedding if for any quantifier-free formula $\phi(\bar{x},\bar{a})$ over $A_0$, if 
$\mathcal{A}\models\forall\bar{x}\phi(\bar{x},\bar{a})$ then $\mathcal{B}\models\forall\bar{x}\phi(\bar{x},f(\bar{a}))$.  

On the light of the above we introduce the notion of a {\em strong $\forall$-Fra\"{i}sse class} by strengthening the $\forall$-AP property. 

\begin{definition}[Strong Universal Fra\"{i}ss\'{e} class]
Let $\mathcal{K}$ be a countable (with respect to isomorphism types) non-empty class of finitely generated $\mathcal{L}$-structures with the following properties:
\begin{itemize}
 \item (IP) the class $\mathcal{K}$ is closed under isomorphisms;
 \item ($\forall$-HP) the class $\mathcal{K}$ is closed under finitely generated $\forall$-substructures;
 \item ($\forall$-JEP) if $\mathcal{A}_1$, $\mathcal{A}_2$ are in $\mathcal{K}$, then there is $\mathcal{B}$ in $\mathcal{K}$ 
 and $\forall$-embeddings $f_i:\mathcal{A}_i\rightarrow_{\forall} \mathcal{B}$ for $i\leq 2$;
 \item (strong $\forall$-AP) if $\mathcal{A}_0$, $\mathcal{A}_1$, $\mathcal{A}_2$ are in $\mathcal{K}$ and $f_i:\bar{a}\rightarrow_{\forall} \mathcal{A}_i$ for $i\leq 2$ are partial $\forall$-maps of 
 some tuple $\bar{a}\in\mathcal{A}_0$, then there is $\mathcal{B}$ in $\mathcal{K}$ and $\forall$-embeddings $g_i:\mathcal{A}_i\rightarrow_{\forall} \mathcal{B}$ for $i\leq 2$ with 
 $g_1\circ f_1(\bar{a})=g_2\circ f_2(\bar{a})$.    
\end{itemize}
Then $\mathcal{K}$ is a strong universal Fra\"{i}ss\'{e} class or for short a strong $\forall$-Fra\"{i}ss\'{e} class.
\end{definition}

We can prove that strong universal Fra\"{i}ss\'{e} limits are $\forall$-homogeneous. 

\begin{theorem}
Let $\mathcal{K}$ be a strong $\forall$-Fra\"{i}ss\'{e} class. Then there exists a countable $\mathcal{L}$-structure $\mathcal{M}$ such that:
\begin{itemize}
 \item the $\forall$-age of $\mathcal{M}$ is exactly $\mathcal{K}$; 
 \item the $\mathcal{L}$-structure $\mathcal{M}$ is $\forall$-homogeneous;
 \item the $\mathcal{L}$-structure $\mathcal{M}$ is the union of a $\Pi_1^0$-chain of $\mathcal{L}$-structures in $\mathcal{K}$.
\end{itemize}
Moreover, any other countable $\mathcal{L}$-structure with the above properties is isomorphic to $\mathcal{M}$.
\end{theorem}

With respect to these notions we will prove that the class of nonabelian limit groups forms a $\forall$-Fra\"{i}ss\'{e} class (see Theorem \ref{LimitFraisse}) and the class of 
abelian limit groups, i.e. finitely generated free abelian groups, forms a strong $\forall$-Fra\"{i}ss\'{e} class (see Theorem \ref{FreeAbelianFraisse}). It is an open question 
whether nonabelian limit groups form a strong $\forall$-Fra\"{i}ss\'{e} class.

 We will skip the proofs of the above results and only prove the theorems of the next subsection since the arguments are identical.

\subsection{Elementary Fra\"{i}ss\'{e} limits}

Let $\mathcal{M}$ be an $\mathcal{L}$-structure. The elementary age of $\mathcal{M}$, $e$-$age(\mathcal{M})$, is 
the class of all finitely generated elementary substructures of $\mathcal{M}$ up to isomorphism. 

When $\mathcal{A}$ is an elementary substructure of $\mathcal{M}$ we denote it by $\mathcal{A}\prec_{e}\mathcal{M}$.

\begin{definition}
Let $\mathcal{A}, \mathcal{B}$ be $\mathcal{L}$-structures and $\bar{a}$ be a tuple from $\mathcal{A}$. Then a map $f:\bar{a}\rightarrow\mathcal{B}$ 
is called partial elementary if for any $\mathcal{L}$-formula $\phi(\bar{x})$, $\mathcal{A}\models\phi(\bar{a})$ if and only if $\mathcal{B}\models\phi(f(\bar{a}))$. 
\end{definition}

\begin{definition}\label{e-Fraisse}
Let $\mathcal{K}$ be a countable (with respect to isomorphism types) non-empty class of finitely generated $\mathcal{L}$-structures with the following properties:
\begin{itemize}
 \item (IP) the class $\mathcal{K}$ is closed under isomorphisms;
 \item (e-HP) the class $\mathcal{K}$ is closed under finitely generated elementary substructures, i.e. if $\mathcal{A}\in \mathcal{K}$ and $\mathcal{B}$ 
 is a finitely generated elementary substructure of $\mathcal{A}$, then $\mathcal{B}\in\mathcal{K}$;
 \item (e-JEP) if $\mathcal{A}_1$, $\mathcal{A}_2$ are in $\mathcal{K}$, then there is $\mathcal{B}$ in $\mathcal{K}$ 
 and elementary embeddings $f_i:\mathcal{A}_i\rightarrow \mathcal{B}$ for $i\leq 2$;
 \item (strong e-AP) if $\mathcal{A}_0$, $\mathcal{A}_1$, $\mathcal{A}_2$ are in $\mathcal{K}$ and $f_i:\bar{a}\rightarrow \mathcal{A}_i$, for $i\leq 2$, are partial elementary maps of 
 some tuple $\bar{a}\in\mathcal{A}_0$,  
 then there is $\mathcal{B}$ in $\mathcal{K}$ and elementary embeddings $g_i:\mathcal{A}_i\rightarrow \mathcal{B}$ for $i\leq 2$ with $g_1\circ f_1(\bar{a})=g_2\circ f_2(\bar{a})$.  
\end{itemize}

Then $\mathcal{K}$ is a strong elementary Fra\"{i}ss\'{e} class or a strong $e$-Fra\"{i}ss\'{e} class for short. 
\end{definition}

\begin{definition}
Let $\mathcal{M}$ be a countable $\mathcal{L}$-structure. Then $\mathcal{M}$ is homogeneous if there exists an automorphism taking the tuple $\bar{a}$ to the tuple $\bar{b}$ 
whenever $tp^{\mathcal{M}}(\bar{a})=tp^{\mathcal{M}}(\bar{b})$.
\end{definition}

\begin{definition}[strong $e$-Extension Property]
%Let $\mathcal{K}$ be an $e$-Fra\"{i}sse class. 
An $\mathcal{L}$-structure has the strong $e$-Extension Property if for any $\mathcal{A}$ finitely generated elementary substructures of $\mathcal{M}$, any $\bar{a}\in\mathcal{A}$ and 
any partial elementary embedding $f:\bar{a}\rightarrow \mathcal{M}$, there exists an elementary embedding $g:\mathcal{A}\rightarrow\mathcal{M}$ that extends $f$.  
\end{definition}

\begin{lemma}\label{HomogeneousExtension}
Let $\mathcal{M}$ be an $\mathcal{L}$-structure which is the union of an elementary chain $\mathcal{M}_1\prec_e\mathcal{M}_2\prec_e\ldots\prec_e\mathcal{M}_n\prec_e\ldots$ of 
finitely generated $\mathcal{L}$-structures. Then $\mathcal{M}$ is homogeneous if and only if it has the strong $e$-Extension Property. 
\end{lemma}
\begin{proof}
First assume that $\mathcal{M}$ is homogeneous. Let $\mathcal{A}\prec_e\mathcal{B}$ be finitely generated elementary substructures of $\mathcal{M}$ and  
$f:\bar{a}\rightarrow \mathcal{M}$ be a partial elementary map, where $\bar{a}$ is a tuple from $\mathcal{A}$. 
Then $tp^{\mathcal{M}}(\bar{a})=tp^{\mathcal{A}}(\bar{a})=tp^{\mathcal{M}}(f(\bar{a}))$. By the homogeneity of $\mathcal{M}$ there exists an automorphism $g$ 
taking $\bar{a}$ to $f(\bar{a})$. The restriction of $g$ on $\mathcal{B}$ is an elementary map extending $f$. Note that this direction does not use that $\mathcal{M}$ 
is the union of an elementary chain.,

For the other direction, assume that $\mathcal{M}$ has the strong $e$-Extension Property and let $tp^{\mathcal{M}}(\bar{a})=tp^{\mathcal{M}}(\bar{b})$. Let $\mathcal{M}_i$ be 
an $\mathcal{L}$-structure in the elementary chain that contains $\bar{a}$. Let $c$ be an element of $\mathcal{M}$ and $\mathcal{M}_j$ be an $\mathcal{L}$-structure in 
the elementary chain that contains both $\bar{a}$ and $c$. We obviously have $\mathcal{M}_i\prec_e\mathcal{M}_j$ and a partial elementary map 
$f:\bar{a}\rightarrow\mathcal{M}$ with $f(\bar{a})=\bar{b}$. By the $e$-Extension Property we get an elementary map $g:\mathcal{M}_j\rightarrow\mathcal{M}$ that extends $f$. 
Thus, we get $tp^{\mathcal{M}}(\bar{a},c)=tp^{\mathcal{M}}(\bar{b}, g(c))$ and we conclude the proof by a back-and-forth argument. 
 
\end{proof}

\begin{theorem}
Let $\mathcal{K}$ be a strong $e$-Fra\"{i}ss\'{e} class. Then there exists a countable $\mathcal{L}$-structure $\mathcal{M}$ such that:
\begin{itemize}
 \item the $e$-age of $\mathcal{M}$ is exactly $\mathcal{K}$; 
 \item the $\mathcal{L}$-structure $\mathcal{M}$ is homogeneous;
 \item the $\mathcal{L}$-structure $\mathcal{M}$ is the union of an elementary chain of $\mathcal{L}$-structures in $\mathcal{K}$.
\end{itemize}
Moreover, any other countable $\mathcal{L}$-structure with the above properties is isomorphic to $\mathcal{M}$.
\end{theorem}

\begin{proof}
We will construct an elementary chain $\mathcal{M}_1\prec_e\mathcal{M}_2\prec_e\ldots\prec_e\mathcal{M}_n\prec_e\ldots$ of $\mathcal{L}$-structures in $\mathcal{K}$ and 
prove that its union $\mathcal{M}:=\bigcup_{i<\omega}\mathcal{M}_i$ has all the desired properties.

Let $\{\mathcal{A}_i\}_{i<\omega}$ be a list of all $\mathcal{L}$-structures in $\mathcal{K}$ (up to isomorphism). Suppose we have constructed $\mathcal{M}_i$. We construct $\mathcal{M}_{i+1}$   
taking cases depending on whether $i$ is even or odd. 
\begin{itemize}
\item If $i=2n$ is even, then we apply $e$-$JEP$ to $\mathcal{M}_i$ and $\mathcal{A}_n$. We set $\mathcal{M}_{i+1}$ to be the resulting $\mathcal{L}$-structure. 
\item If $i=2n+1$ is odd, then we apply strong $e$-$AP$ to $f:\bar{a}\rightarrow\mathcal{M}_i$, $Id:\bar{a}\rightarrow\mathcal{A}$, for some finitely generated $\mathcal{A}\prec_e\mathcal{M}_i$, some $\bar{a}\in\mathcal{A}$ 
and some partial elementary map $f$, to be specified later. We set $\mathcal{M}_{i+1}$ to be the resulting $\mathcal{L}$-structure.
\end{itemize}

The even cases take care of the equality between the elementary age of $\mathcal{M}$ and $\mathcal{K}$. 

In order to prove the strong $e$-Extension Property and consequently the 
homogeneity of $\mathcal{M}$ we need to ensure that for any finitely generated elementary substructure $\mathcal{A}$ of $\mathcal{M}$ and any partial 
elementary map $f:\bar{a}\rightarrow \mathcal{M}$ there exists an elementary embedding $g:\mathcal{A}\rightarrow \mathcal{M}$ that extends $f$. We may assume 
that $f(\bar{a})$ and $\mathcal{A}$ are in $\mathcal{M}_j$ for some $j<\omega$. Thus, we need to ensure that for some odd $i\geq j$, we have chosen the 
partial map $f$ to apply the strong $e$-$AP$. This is clearly possible since there are only countably many choices 
of partial elementary maps. 

The uniqueness result follows from an easy back-and-forth argument.

%Let $\{f_i \ | \ f_i:\bar{a}\rightarrow \mathcal{A}_j, \bar{a}\in\mathcal{A}_k\}_{i<\omega}$ be a list of all 
%partial elementary maps for all couples $\mathcal{A}_k\prec_e\mathcal{A}_j$ in $\mathcal{K}$. We note that both sets are countable. We will 
%define the $\mathcal{M}_i$'s recursively. We set  
%$\mathcal{M}_1:=\mathcal{A}_1$. Suppose we have defined $\mathcal{M}_n$, we define $\mathcal{M}_{n+1}$ as follows: 
%\begin{itemize}
% \item if $n=2k$, then to get $\mathcal{M}_{n+1}$, we apply ($e$-JEP) to $\mathcal{M}_n$ and $\mathcal{A}_{k+1}$;
% \item if $n=2k+1$, then to get $\mathcal{M}_{n+1}$ we apply ($e$-AP) to a couple $\mathcal{A}\prec_e\mathcal{B}$, 
% where $\mathcal{A}$ is an elementary substructure of $\mathcal{M}_n$, and a partial elementary map $f:\bar{a}\rightarrow\mathcal{M}_n$.
%\end{itemize}
%It is straightforward to check that the union of an elementary chain constructed as above has property $(*)$. Thus, $e$-$age(\mathcal{M})=\mathcal{K}$ 
%and $\mathcal{M}$ has the $e$-Extension Property which in turn implies, by lemma \ref{HomogeneousExtension}, that $\mathcal{M}$ is homogeneous.
\end{proof}
   
 \subsection{Finitary extensions} \label{sec:2.4}
There is a straightforward generalization of the classical Fra\"{i}ss\'{e} result when we can only define ``special'' extensions of structures in a finitary way. A particular case 
is when special extensions can only be defined for say finite $\mathcal{L}$-structures. In this case it is not clear a priori how to give meaning to the notion of age of a countably infinite $\mathcal{L}$-structure 
with respect to special embeddings.   
In this subsection we record well known results that deal with this case. 

Let $\mathcal{K}$ be a class of finitely generated $\mathcal{L}$-structures. Suppose $\mathcal{A}\subseteq\mathcal{B}$ are in $\mathcal{K}$, we define a notion of special extension and denote it by 
$\mathcal{A}\sqsubseteq \mathcal{B}$. The notion of a special embedding between structures in $\mathcal{K}$ is defined in the natural way. 
If the $\mathcal{L}$-structure $\mathcal{M}$ is the union of a $\sqsubseteq$-chain  $M_1\sqsubseteq M_2\sqsubseteq\ldots\sqsubseteq M_n\sqsubseteq\ldots$ we will 
write $A\sqsubseteq M$ if $A\sqsubseteq M_i$ for some $i<\omega$. We will assume that $\sqsubseteq$ satisfies the following properties: 

For any $\mathcal{A}$, $\mathcal{B}$, $\mathcal{C}$ in $\mathcal{K}$:
\begin{itemize}
 \item (N1) the $\mathcal{L}$-structure $\mathcal{A}$ is a special extension of itself, i.e. $\mathcal{A}\sqsubseteq \mathcal{A}$;
 \item (N2) if $\mathcal{A}\sqsubseteq\mathcal{B}\sqsubseteq\mathcal{C}$, then $\mathcal{A}\sqsubseteq\mathcal{C}$; 
 \item (N3) if $\mathcal{A}\sqsubseteq\mathcal{C}$ and $\mathcal{A}\subseteq\mathcal{B}\subseteq\mathcal{C}$, then $\mathcal{A}\sqsubseteq\mathcal{B}$. 
\end{itemize}

The last property, (N3), guarantees that $\mathcal{A}\sqsubseteq\mathcal{M}$, as defined above, does not depend on the choice of the $\sqsubseteq$-chain 
$M_1\sqsubseteq M_2\sqsubseteq\ldots\sqsubseteq M_n\sqsubseteq\ldots$.

\begin{definition}\label{specialFraisse}
Let $\mathcal{K}$ be a countable (with respect to isomorphism types) non-empty class of finitely generated $\mathcal{L}$-structures that has countably many $\sqsubseteq$-embeddings between any pair of elements, 
with the following properties:
\begin{itemize}
 \item (IP) the class $\mathcal{K}$ is closed under isomorphisms;
 \item ($\sqsubseteq$-HP) the class $\mathcal{K}$ is closed under finitely generated elementary $\sqsubseteq$-substructures, i.e. if $\mathcal{A}\in \mathcal{K}$ and $\mathcal{B}$ is a finitely generated $\mathcal{L}$-structure such that 
 $\mathcal{B}\sqsubseteq \mathcal{A}$, then $\mathcal{B}\in\mathcal{K}$;
 \item ($\sqsubseteq$-JEP) if $\mathcal{A}_1$, $\mathcal{A}_2$ are in $\mathcal{K}$, then there is $\mathcal{B}$ in $\mathcal{K}$ 
 and $\sqsubseteq$-embeddings $f_i:\mathcal{A}_i\rightarrow \mathcal{B}$ for $i\leq 2$;
 \item ($\sqsubseteq$-AP) if $\mathcal{A}_0$, $\mathcal{A}_1$, $\mathcal{A}_2$ are in $\mathcal{K}$ and $f_i:\mathcal{A}_0\rightarrow \mathcal{A}_i$, for $i\leq 2$, are $\sqsubseteq$-embeddings,  
 then there is $\mathcal{B}$ in $\mathcal{K}$ and $\sqsubseteq$-embeddings $g_i:\mathcal{A}_i\rightarrow \mathcal{B}$ for $i\leq 2$ with $g_1\circ f_1=g_2\circ f_2$.  
\end{itemize}

Then $\mathcal{K}$ is a $\sqsubseteq$-Fra\"{i}ss\'{e} class. 
\end{definition}

Suppose $\mathcal{M}$ is the union of a $\sqsubseteq$-chain  $M_1\sqsubseteq M_2\sqsubseteq\ldots\sqsubseteq M_n\sqsubseteq\ldots$ of finitely generated $\mathcal{L}$-structures. Then the $\sqsubseteq$-age of 
$\mathcal{M}$, $\sqsubseteq$-$age(\mathcal{M})$, is the class of all $\sqsubseteq$-substructures of $\mathcal{M}$ up to isomorphism.

\begin{theorem}\label{th:2.14}
Let $\mathcal{K}$ be a $\sqsubseteq$-Fra\"{i}ss\'{e} class. Then there exists a countable $\mathcal{L}$-structure $\mathcal{M}$ such that:
\begin{itemize}
 \item the $\mathcal{L}$-structure $\mathcal{M}$ is the union of a $\sqsubseteq$-chain of $\mathcal{L}$-structures in $\mathcal{K}$;
 \item the $\sqsubseteq$-age of $\mathcal{M}$ is exactly $\mathcal{K}$; 
 \item the $\mathcal{L}$-structure $\mathcal{M}$ is $\sqsubseteq$-homogeneous, i.e. whenever $\mathcal{A}$, $\mathcal{B}$ are isomorphic finitely generated $\sqsubseteq$-substructures 
 then there is an automorphism of $\mathcal{M}$ extending this isomorphism;
\end{itemize}
Moreover, any other countable $\mathcal{L}$-structure with the above properties is isomorphic to $\mathcal{M}$.
\end{theorem}

Note, that Theorem \ref{forallFraisse} is a special case of the theorem above.

All definitions and results in this subsection go through in the category of relatively finitely generated $\mathcal{L}$-structures, i.e. when all $\mathcal{L}$-structures considered  
are extensions of a fixed $\mathcal{L}$-structure $\mathcal{N}$ and finitely generated with respect to it, moreover all special embeddings between them fix $\mathcal{N}$ point-wise.  

%if $\mathcal{A}$, $\mathcal{B}$ are in $\mathcal{K}$, $\mathcal{A}\sqsubseteq \sqsubseteq \mathcal{M}$ and $f:\mathcal{A}\rightarrow\mathcal{B}$ is a $\sqsubseteq$-embedding, 
% then there is a $\sqsubseteq$-embedding $g:\mathcal{B}\rightarrow\mathcal{M}$ such that $g\circ f(a)=a$ for all $a\in\mathcal{A}$ ;

\section{The class of limit groups}

Limit groups have been introduced by Sela in \cite{Sela1}. It turned out that the class of limit groups coincides with the long studied class of finitely generated fully residually 
free groups introduced by Baumslag in \cite{B}. A group $G$ is {\em fully residually free} if for any finite subset $X$ of $G$ there exists a homomorphism from $G$ to a free group that 
is injective on $X$. In \cite{KM} (see also \cite{Sela1}) it was proved that limit groups are finitely presented, thus the class of limit groups is countable. 
One can characterize limit groups in terms of first-order logic: they are the finitely generated models of the universal theories of free groups (abelian and nonabelian). 
The class of abelian limit groups is, in particular, the class of finitely generated free abelian groups. 

Limit groups do not form a Fra\"{i}ss\'{e} class because, although they satisfy the $HP$ and $JEP$ conditions, they do not satisfy the $AP$ condition. To see this consider the groups 
$\Z:=\langle z\rangle$, $\F_2:=\langle a,b\rangle$, $\F_2':=\langle c,d\rangle$ and the embeddings $f_1:\Z\rightarrow \F_2$, given by $z\mapsto a^2b^2$, and $f_2:\Z\rightarrow \F_2'$, given by $z\mapsto c^2$. 
The amalgamated free product $\F_2*_{\Z}\F_2'$ using the maps $f_1, f_2$ is the free product $\langle a, b, c \  | \ a^2b^2=c^2 \rangle * \langle d\rangle$. This free product is not a limit group 
since $\langle a,b, c \  | \ a^2b^2=c^2 \rangle$ is not a limit group. The latter is true because in limit groups, by a result of Lyndon for free groups, whenever the product of three squares is trivial, then the elements of the product commute. 
Now we can see that the class of limit groups does not have $AP$, 
since for any group $G$ and any $g_1:\F_2\rightarrow G$, $g_2:\F_2'\rightarrow G$ such that $g_1\circ f_1=g_2\circ f_2$, we get that $g_1\circ f_1(z)= g_1(a)^2g_1(b)^2=g_2(c)^2=g_2\circ f_2(z)$. But the latter 
implies that $g_1(a)$, $g_1(b)$, $g_2(c)$ commute in $G$, hence $a$, $b$, commute in $\F_2$, a contradiction.

\subsection{Free Abelian groups}

In this subsection we show that the class of finitely generated free abelian groups forms a strong $\forall$-Fra\"{i}ss\'{e} class. We remark that in this class the notions of  
pure subgroup, existentially closed subgroup and direct factor coincide. All properties follow rather easily, therefore we only give the details for proving the strong $\forall$-AP. 

\begin{lemma}\label{Pure}
Let $\bar{b}\in \Z^n$ and $\bar{c}\in \Z^m$ such that $tp_{\forall}^{\Z^n}(\bar{b})=tp_{\forall}^{\Z^m}(\bar{c})$. Then, the isomorphism between $\langle\bar{b}\rangle$ 
and $\langle\bar{c}\rangle$ extends to the smallest pure subgroups that contain $\bar{b}$ and $\bar{c}$ respectively. 
\end{lemma}

\begin{proposition}
The class of finitely generated free abelian groups has the strong $\forall$-AP property. 
\end{proposition}

\begin{proof}
Let $\bar{a}$ be a tuple from $\Z^m$ and $f_i:\bar{a}\rightarrow_{\forall} \Z^{m_i}$ be partial $\forall$-maps for $i=1,2$. Then $tp_{\forall}^{\Z^{m_1}}(\bar{b})=tp_{\forall}^{\Z^{m_2}}(\bar{c})$ where $\bar{b}=f_1(\bar{a})$ and 
$\bar{c}=f_2(\bar{a})$. By the previous lemma there is a direct factorization of $\Z^{m_1}$ as $B\times \Z_1$ and one of $\Z^{m_2}$ as $C\times \Z_2$ such that 
$f:C\cong B$ with $f(\bar{c})=\bar{b}$. Consider the group $D= \Z_1\times B \times \Z_2$ and the embeddings $g_1=Id$ and $g_2=(f,Id)$, where $f$ is the isomorphism from $C$ to $B$. 
The embeddings are in particular $\forall$-embeddings and moreover $g_1\circ f_1(\bar{a})= g_2\circ f_2(\bar{a})$. 
\end{proof}

Thus, we have the following theorem. 

\begin{theorem}\label{FreeAbelianFraisse}
There exists a countable group $G$ with the following properties: 
\begin{itemize} 
 \item the $\forall$-$age$ of $G$ is the class of finitely generated free abelian groups;
 \item the group $G$ is  $\forall$-homogeneous; 
 \item the group $G$ is the union of a $\forall$-chain of finitely generated free abelian groups.
\end{itemize}
Moreover, any countable group with the above properties is isomorphic to $G$.
\end{theorem}

It is not hard to see in this case that the strong $\forall$-Fra\"{i}ss\'{e} limit is $\Z^{(\omega)}$, the direct sum of $\omega$ copies of $\Z$. 
%Moreover, in contrast to the classical Fra\"{i}ss\'{e} theory, the second condition is necessary for concluding the uniqueness of the limit amongst countable groups. 
%Indeed, the direct product $\Z^{\omega}$ of $\omega$ copies of $\Z$ satisfy both other conditions but it is not isomorphic to $\Z^{(\omega)}$. 

As a matter of fact the class of finitely generated free abelian groups forms a Fra\"{i}ss\'{e} class. In this case the Fra\"{i}ss\'{e} limit is $\mathbb{Q}^{(\omega)}$, 
the direct sum of $\omega$ copies of $\mathbb{Q}$.

\subsection{Nonabelian limit groups}

In this subsection we will make use of some basic constructions in combinatorial group theory, in particular amalgamated free products and HNN-extensions. We refer the reader 
to the classical books \cite{Magnus}, \cite{LySch}.

Nonabelian limit groups in contrast to the abelian limit groups do not form a Fra\"{i}ss\'{e} class. In particular the $HP$ condition fails.

We now prove that the class $\mathcal{F}$ of nonabelian limit groups forms a $\forall$-Fra\"{i}ss\'{e} class. The $\forall$-HP condition is rather obvious, since a finitely generated 
subgroup of a limit group is itself a limit group, in addition an existentially closed subgroup of a nonabelian group cannot be abelian. 

It will be useful to remark that in the class of limit groups one can characterize $\forall$-embeddings in the following way:

\begin{lemma}{\cite[Proposition 3]{KMeq}}\label{Retraction}
Let $L, M$ be limit groups. Then $L\leq_{\forall} M$ if and only if for every finite subset $X\subset M$, there exists a retraction $r_X:M\rightarrow L$ that is injective on $X$.   
\end{lemma}

We also recall that abelian subgroups of limit groups are finitely generated and thus in particular they are free abelian (cf. \cite[Proposition 1]{KMeq}).

For proving the $\forall$-JEP and $\forall$-AP conditions we will use the following theorems. We recall that an {\em extension of a centralizer} of a group $G$ is 
the group $\langle G, t \ | \ [c, t]=1, c\in C_{G}(u)\rangle$. A {\em finite iterated extension of centralizers} of a group is obtained by finitely many applications of the previous 
construction.

\begin{theorem}[\cite{KM}]\label{CentralizerExtension}
Let $L$ be a limit group (respectively nonabelian limit group). 
Then $L$ embeds into some finite iterated extension of centralizers of a free group (respectively nonabelian free group).   
\end{theorem}

For notation purposes 
when $K$ is a subgroup of both $L$ and $M$ we call a map from $L$ to $M$ a $K$-map if it is the identity on $K$.

\begin{theorem}[\cite{KhMyasn:2012}]\label{CentralizerExtensionParameters}
Let $L, M$ be limit groups and $L\leq_{\forall} M$. Then $M$ $L$-embeds into some finite iterated extension of centralizers of $L$. 
\end{theorem}

\begin{lemma}\label{forallCentralizers}
Let $L$ be a limit group and $M$ be a finite iterated extension of centralizers of $L$. Then $L\leq_{\forall} M$.
\end{lemma}

We first recall the following well known result \cite{MR}:

\begin{lemma}\label{BigPowers}
Suppose in a nonabelian free group $\F$ an element $g$ has the following form:  
$$ a_0 b^{i_1} a_1 b^{i_2} a_2\ldots a_{n-1} b^{i_n} a_n$$
where $n\geq 1$ and $[a_i , b] \neq 1$ for all $0<i<n$. Then $g \neq 1$ whenever
$min_k |i_k |$ is sufficiently large.
\end{lemma}

The above property easily passes to nonabelian limit groups using the fact that they are fully residually free.
\\ \\
\noindent {\em Proof of Lemma \ref{forallCentralizers}:}

It is enough to show that if $L$ is a limit group and $M$ is an extension of centralizers, then $L\leq_{\forall}M$. If $L$ is (free) abelian, then $M$ is free abelian as well and has $L$ as a direct factor. 
We may therefore assume that $L$ is nonabelian. Using Lemma \ref{Retraction} we need to find for each finite set $X\subseteq M$ a retraction $r_X:M\rightarrow L$ such that $r_X$ is injective on $X$. 
Since $M$ is an extension of centralizers of $L$, it is in particular an HNN-extension $L*_C$ where $C$ is a maximal abelian subgroup of $L$.   

It follows from the normal form theorem for HNN extensions that any element of $M$ is either trivial or of the form $a_0t^{\epsilon_1}a_1\ldots a_{n-1}t^{\epsilon_n}a_n$ 
where $a_i$ is in $L\setminus C$ for all $0<i<n$ and $\epsilon_i\in\Z\setminus\{0\}$ for all $i\leq n$. Now consider the retractions $r:M\rightarrow L$, that fix $L$ and send the stable letter $t$ to high powers of some 
generator of the free abelian group $C$. If $X$ is a singleton the result follows easily from Lemma \ref{BigPowers}. In the case $|X|> 1$, one has to be more careful but still the result is 
an easy exercise on applying Lemma \ref{BigPowers}. \qed    
\ \\ \\
We first prove the $\forall$-JEP condition.  

\begin{proposition}\label{forallJEP}
Let $L, M$ be limit groups and $L$ be nonabelian. Then $L\leq_{\forall} L*M$.  
\end{proposition}

\begin{proof}
It is not hard to see, using the same arguments as in the proof of Lemma \ref{forallCentralizers}, that $L\leq_{\forall} L*\F$, where $\F$ is a free group. 
Now by Theorem \ref{CentralizerExtension}, $M$ embeds in a finite iterated extension of centralizers of some free group $\F$, say $\Gamma$.  
In addition, $L*\F\leq_{\forall} L*\Gamma$ as $L*\Gamma$ can be obtained as a finite iterated extension of centralizers of $L*\F$. Therefore, $L\leq_{\forall}L*\Gamma$. 
Now observe that since $L*M$ is a subgroup of $L*\Gamma$ that contains $L$ we have that $L\leq_{\forall}L*M$, proving the result. 

\end{proof}

It is easy to see that the free product of two limit groups is a limit group. Lemma \ref{forallJEP} in addition 
proves that a nonabelian limit group is existentially closed in the free product of itself with another limit group, thus making  
the $\forall$-JEP condition true for the class of nonabelian limit groups. 

  Notice that if $K\leq_{\forall} L, c\in K$ and $C=C_K(c)$, $C_1=C_L(c)$, then $C_1=C\oplus C_2$ for some finitely generated (maybe trivial) free abelian group $C_2$. In other words, $C$ is a direct summand in $C_1$. 
  Indeed, otherwise some elements in $C$ will have roots in $L$ but not in $K$.

\begin{proposition}\label{forallAP}  Class $\mathcal F$ has $\forall$-AP property.
 \end{proposition}
\begin{proof}
Let $K, L, M$ be nonabelian limit groups such that $K\leq_{\forall} L$ and $K\leq_{\forall} M$.  We will construct a limit group $\Gamma$ that is a maximal 
commutative transitive quotient of the amalgamated product $L*_KM$  and show that $L,M\leq_{\forall} \Gamma .$

 By Theorem \ref{CentralizerExtensionParameters}, $L$ $K$-embeds in  finite iterated extension of centralizers of $K$, say $L_m$, and  $M$ also $K$-embeds in  finite iterated extension of centralizers of $K$, say $M_n$.

Let $K=L_0<L_1<\ldots ,L_m,$ where $L _i=\langle L _{i-1}, s_i|[C_{L _{i-1}}(c_i),s_i]=1\rangle ,$

$K=M_0<M_1<\ldots< M_n, $ where $M _i=\langle M _{i-1}, t_i|[C_{M _{i-1}}(d_i),t_i]=1\rangle .$ 

%We can assume without loss of generality that $c_1,\ldots ,c_q\in K, q\leq m$, $q$ is the maximal such number and if $c_i$ is conjugate in K  into $C_K(c_j)$, $i,j\leq q,$ then  we replace $C_{L_{j-1}}(c_j)$ by the corresponding conjugate and take $c_j=c_i$. We also assume  $d_1,\ldots ,d_r\in K, r\leq n$ and satisfy similar assumptions. In addition, if a conjugate of $C_K(d_j)$ and $C_K(c_i)$ have non-trivial intersection, then we assume that $c_i=d_j$ ($i\leq q, j\leq r.$)

Define now $N_{0j}=M_j,\  N_{i0}=L_i$  and  
$N_{ij}=\langle N_{i(j-1)}, t_j|[C_{N_{i(j-1)}}(d_j),t_j]=1\rangle .$
We will prove by induction on $m+n$ that $L_m,M_n$ $K$-embed into $N_{mn}$.
(It can be also shown that $N_{ij}=\langle N_{(i-1)j}, s_i|[C_{N_{(i-1)j}}(c_i),s_i]=1\rangle $ although this is not needed for the proof.)

For the base case first suppose $m=n=1.$ 
We have two cases:

1) $C_K(c_1)$ and $C_K(d_1)$ are not conjugate in $K$. Then $$N_{11}= \langle K,s_1,t_1| 
[C_K(c_1), s_1]=1,  [C_K(d_1), t_1]=1\rangle .$$

2)  $C_K(c_1)$ and $C_K(d_1)$ are conjugate in $K$, then    
$c_1=d_1^g, g\in K.$  Then $$N_{11}= \langle K,s_1,t_1| 
[C_K(c_1), s_1]=1,  [C_K(c_1), t_1^g]=1, [t_1^g,s_1]=1\rangle .$$

In both cases $N_{11}$ is obtained by an extension of a centralizer from $L_1=N_{10}$ and also from $M_1=N_{01}$, therefore $L_1,M_1$ $K$-embed  into $N_{11}.$  

Similarly,   since $N_{1,j}=\langle N_{1(j-1)}, t_j| [C_{N_{1(j-1)}}(d_j),t_j]=1\rangle$, we have that 
$M_j$ $K$-embeds into $N_{1j}$ and $L_1,N_{1(j-1)}$ $L_1$-embed into $N_{1j}.$ 
The sum of the lengths of the sequences of  centralizer extensions  $L_1 \to N_{11} \to  \ldots \to N_{1n}$ and $L_1 \to L_2 \to \ldots \to L_m$  is smaller than $n+m$.
Therefore, by induction,  $L_m, N_{1n}$ $L_1$-embed into $N_{mn}$.  Then $L_m, M_n$ $K$-embed into $N_{mn}.$

We have $L,M\leq N_{mn}$. Let $\Gamma$ be the subgroup of $N_{mn}$ generated by $L$ and $M$.  We will show that $L,M\leq_{\forall}\Gamma .$
We use induction on $n$. For the base case first suppose $m=n=1.$  Then $\Gamma \leq N_{11}.$

We take any finite set of non-trivial elements in $\Gamma$ in normal form corresponding to $N_{11}$ 
$$S=\{\ell _{i1}{t_1}^{\alpha _{i1}}\ldots \ell _{ik_i}{t_1}^{\alpha _{ik_i}}, i\in I\},$$
where $\ell _{ij}\in L,  [\ell _{ij},d_1]\neq 1$,$\alpha _{ij}\neq 0$  for any pair $i,j$ except that, maybe, $\ell _{i1}=1$ and , maybe, $\alpha _{ik_i}=0.$

By Lemma \ref{BigPowers} for limit groups any homomorphism $\phi :N_{11}\rightarrow N_{10}=L_1$ that is identical on $N_{1,0}$ and  
maps $t_1$ to a sufficiently big power of  $d_1$ will map all elements in $S$ to nontrivial elements. The restriction $\phi_{\Gamma}$ is a discriminating retraction on $L$.  
Therefore by Lemma \ref{Retraction} $L\leq_{\forall}\Gamma$. By symmetry $M\leq_{\forall}\Gamma$.

Similarly, for arbitrary $m$, suppose that $\Gamma\leq N_{m1}$. We can again consider an arbitrary set $S$ of nontrivial elements in $\Gamma$. It will have the same form as in the case $m=1$. 
Any homomorphism $\phi :N_{m1}\rightarrow N_{m0}=L_m$ that is identical on $L_m$ and  
maps $t_1$ to a sufficiently large power of $d_1$ will map all elements in $S$ to nontrivial elements. 
The restriction $\phi_{\Gamma}$ is a discriminating retraction on $L$. Therefore by Lemma \ref{Retraction} $L\leq_{\forall}\Gamma$.

Suppose that we proved the statement for arbitrary $m$ and for $n-1$ and $\Gamma\in N_{mn}$.  For any finite set $S$ of elements in $\Gamma$, any homomorphism $\phi :N_{mn}\rightarrow N_{m(n-1)}$ 
that is identical on $N_{m(n-1)}$ and  maps $t_n$ to a sufficiently large power of $d_n$
 will map all elements in $S$ to nontrivial elements. We can now take $\phi (M)$ instead of $M$ and apply induction.
\end{proof}

Finally, combining the results above we get:

\begin{theorem}\label{LimitFraisse}
Class $\mathcal{F}$ is a $\forall$-Fra\"{i}ss\'{e} class. 

In particular there exists a countable group $G$ with the following properties: 
\begin{itemize}
 \item the $\forall$-age of $G$ is the class $\mathcal{F}$;
 \item the group $G$ is weakly $\forall$-homogeneous;
 \item the group $G$ is a union of a $\forall$-chain of nonabelian limit groups. 
 \end{itemize}
Moreover, any countable group with the above properties is isomorphic to $G$. 
\end{theorem}

We next show that the $\forall$-Fra\"{i}ss\'{e} limit of the class of nonabelian limit groups is a fully residually free group. 

\begin{proposition}
Let $G$ be a union of a $\forall$-chain of limit groups. Then $G$ is fully residually free. 
\end{proposition}
\begin{proof}
Let $G$ be the union of the following $\forall$-chain $G_1\leq_{\forall} G_2\leq_{\forall}\ldots\leq_{\forall}G_n\leq\ldots$.
Let $X$ be a finite subset of $G$. Then $X$ is a subset of $G_n$ for some $n$. Choose $f:G_n\rightarrow  X$ injective on $X$. Any retraction $r_{\ell+1}:G_{\ell +1}\rightarrow G_{\ell}$ for $\ell >n$ is injective on $X$. 
Let  $f_m:G_m\rightarrow \F$ be $f_m= f\circ r_{n+1}\ldots \circ{r_m}$ for all $m\geq n$. 
The union $f:=\bigcup f_m$ is a morphism from $G$ to $\F$ that is injective on $X$.

\end{proof}

\section{Lyndon's completions}

\subsection{Hierarchy of  extensions of centralizers}

In this section we consider $G$-groups, i.e. extensions of $G$, where $G$ is a countable nonabelian torsion-free CSA-group. 
We relate the free Lyndon $\Z[t]$-group $G^{Z[t]}$ with Fra\"{i}ss\'{e}  limits with respect to extensions of centralizers.  
 
We denote by $ICE(G)$ the class of ($G$)-groups obtained from the given group $G$ as finite iterated extensions of centralizers of $G$. We allow the empty sequence of centralizer extensions, 
so that $G\in ICE(G)$.  Recall that if $G$ is CSA, then  ICE($G$) is also CSA. We say that an embedding  $\phi: A \to B$ is an $ICE$-embedding if $B$ can be obtained from $\phi(A)$ as a finite iterated extension of centralizers. 
In this event we write $\to_{ICE}$.  We will apply the results of Section \ref{sec:2.4} with $\K$ to be $ICE(G)$ and $\sqsubseteq$ to be $ICE$-extensions.

Denote $G(u,t)=\langle G,t|[C_G(u),t]=1\rangle .$

\begin{lemma}\label{le:4.1}
Let $G$ be a countable non-abelian torsion-free CSA group (not necessary finitely generated). Then the class $ICE(G)$ satisfies JEP and AP relative to ICE-embeddings.
\end{lemma}
\begin{proof}
Since $G$ is ICE-embedded in every group from $ICE(G)$ it suffices to prove AP. 

Suppose that $A,B,C \in ICE(G)$ and $\phi:C \to A$ and $\psi:C \to B$ are ICE-embeddings. 

\medskip \noindent
{\em Case 1}. Suppose $A$ and $B$ are obtained from $C$ by a single centralizer extension, i.e., $A = C(u_1,t_1), B = C(u_2,t_2)$. If the centralizers $C(u_1)$ and $C(u_2)$ are conjugate in $C$, say $d^{-1} C(u_1) d = C(u_2)$, 
then the groups $A$ and $B$ are isomorphic with the isomorphism $\alpha: A \to B$ such that $\alpha_C = id_C$ and $t_1 \to d^{-1}t_2d$. If $C(u_1)$ and $C(u_2)$ are not conjugate in $C$ then   
$$A(u_2,t_2) \simeq C(u_1,t_1)(u_2,t_2) \simeq C(u_2,t_2)(u_1,t_1) \simeq B(u_1,t_1),
$$
hence the group $A(u_2,t_2)$ (as well as $B(u_1,t_1)$) gives the required amalgamation.

\medskip \noindent
{\em Case 2}. Suppose now that $A$ and $B$ are  obtained from $C$ by some  finite sequences of CE, so $$C = A_0 \leq_{CE} A_1 \leq_{CE} \ldots \leq _{CE} A_k = A,$$
$$C = B_0 \leq_{CE} B_1 \leq_{CE} \ldots \leq_{CE} B_m = B$$
for some $k,m$. We will use induction on the sum $k+m$. By Case 1 we can amalgam $A_1, B_1$ into some $D_1$, and then amalgam $A_2$ and $D_1$ into some $D_2$,  after $k$ steps we will get the 
sequence of CE:  $B_1 \to D_1 \to  \ldots \to D_k$. The sum of the lengths of the sequences of SE $B_1 \to D_1 \to  \ldots \to D_k$ and $B_1 \to B_2 \to \ldots \to B_m$  is smaller, so by induction on $k+m$ one gets the required amalgamation.

\end{proof}
\begin{lemma} \label{le:4.2} Let $G$ be a countable nonabelian torsion-free CSA group.

 1) (ICE-HP) If $G\rightarrow _{ICE}A, G\leq B, B\rightarrow _{ICE}A,$ then $G\rightarrow _{ICE}B$.

2) (ICE-N3) If $A\leq B \leq C\in ICE (G)$ and $A\rightarrow _{ICE} C,$ then $A\rightarrow _{ICE} B.$
\end{lemma}
\begin{proof} 1) We can assume that both $A$ and $B$ are freely indecomposable relative to $G$. Suppose $$G = A_0 \leq_{CE} A_1 \leq_{CE} \ldots \leq _{CE} A_k = A$$ and on each step the corresponding centralizer is extended by one letter. 
We will use induction on $k$. 

 If $k=1$, $A=A_1=\langle G, t_1| [u_1,t_1]=1\rangle$. Either $B=G$ or  $G$ is a proper subgroup of $B$. In the latter case if $B$ does not contain  the stable letter $t_1$, then $A$  
 must be obtained from $B$ by freely extending the centralizer of $u_1$ by $t_1$.  But in this case $B$ cannot have elements containing $t_1$ in the normal form in $A$.  
 Therefore this case is impossible.  Hence if $G$ is a proper subgroup of $B$, then $B=A$.
  
Consider now the general case. Let $A_{j+1}=\langle A_j,t_j| [C(u_{j+1}),t_{j+1}]=1\rangle$ for $j\leq k-1.$ Suppose that it is  possible to change the order of centralizer extensions such a way that $B$ is conjugate into $A_{k-1}$. 
Since $G\leq B$, $B\leq A_{k-1}.$  
Then by mapping $t_k$ into $u_k$ we see that $A_{k-1}$ can be obtained by a sequence of centralizer extensions from  $B$. By induction, $B$ is obtained from $G$ by a sequence  of extensions of centralizers.  

If it is not  possible to change the order of centralizer extensions such a way that $B$ is conjugate into $A_{k-1}$, then every conjugate $B^g$ contains elements that have $t_k$ in the normal form, 
therefore it must contain the conjugate $t_k^g$ (because to obtain $A_k$ from $B$ we can only use centralizer extensions). If $t_k$ belongs to the centralizer of 
some $u_k\in G$ we can consider $\langle G, t_k|[u_k,t_k]=1\rangle$ instead of $G$ and apply induction. Otherwise, we map $t_k$ to $u_k$. Then the image of $A_k$ is $A_{k-1}$ and the image of $B$ is some subgroup $\bar B$. 
And $A_{k-1}$ is obtained from $\bar B$ by a series of centralizer extensions. Then by induction, $\bar B$ is obtained from $G$ by a  series of centralizer extensions, and $B$ is obtained from $\bar B$ by extending some centralizer by $t_k$. 

2) Let $$A = A_0 \leq_{CE} A_1 \leq_{CE} \ldots \leq _{CE} A_k = C.$$  Then $B\in A_i$ for some $i\leq k$.  Taking $A_i$ instead of $C$ we can assume that $i=k$. 
We use induction on $k$. If $k=1$, then either $B=A$ or $B$ contains elements with $t_1$ in the normal form. 
But since $B\in ICE(G)$ it must be $B=C$ in the latter case. For the general case we can suppose that  $B\not \leq A_{k-1}$. 
If there are elements in $B$ that contain $t_{k}$ in the normal form, then $B$ contains $t_{k} $ because $B$ is obtained from $G$ by centralizer extensions. Therefore
$B=\langle B_1, t_k|[C(u_k)\cap B,t_k]=1\rangle$, where $B_1\in A_{k-1}.$ By induction, $A\rightarrow _{ICE}B_1$, and, therefore,  $A\rightarrow _{ICE}B.$

\end{proof}

If $H$ is the direct limit of a chain of ICE-embeddings of groups from $ICE(G)$, namely
$$G=G_0\leq _{CE}G_1\leq _{CE}\ldots \leq _{CE}G_i\leq _{CE}\ldots ,$$
then for $A\in ICE(G)$ we can define $A\to_{ICE} H$ if $A\to_{ICE} G_i$ for some $i$. It follows from Lemmas
\ref{le:4.1}, \ref{le:4.2}  and the results of Section \ref{sec:2.4} that this definition does not depend on the chain. 
It also follows that $A\to_{ICE} H$ if $H$ can be obtained as the direct limit of a chain of ICE-embeddings of groups from
$ICE(A)$. It also follows that $ICE(G)$ is a Fra\"{i}ss\'{e} class with respect to $ICE$-embeddings. Theorem \ref{th:2.14} implies the following

\begin{theorem} \label{th:Fraisse-ICE}
Let $G$ be a countable non-abelian torsion-free CSA group. Then there exists a countable group $H$ with the following properties: 
\begin{itemize}
 \item the group $H$ is the union of an $ICE$-chain of groups in $ICE(G)$;  
 \item the $ICE$-$age$ of $H$ is exactly $ICE(G)$;
  \item Any  $G$-isomorphism between two groups $A,B \leq_{ICE} H$ can be extended to an automorphism of $H$.
 \end{itemize} 
Moreover, any other countable group with the above properties is $G$-isomorphic to $G$. 
 \end{theorem}

\begin{theorem} \label{th:Lyndon} The group $H$ in Theorem \ref{th:Fraisse-ICE} is isomorphic to the Lyndon's completion $G^{\Z[t]}$ of the group $G$. 
In particular when $G$ is a nonabelian free group $\F$, then $H$ is isomorphic to $\F^{\Z[t]}$.
\end{theorem}

\begin{proof} In \cite{MR} it was proved that the Lyndon's completion $G^{\Z[t]}$ of the group $G$ can be constructed as follows. 
We begin with $G$ and extend all non conjugate centralizers obtaining the group $\tilde G_1$. 
Then we extend all non-conjugate centralizers in $\tilde G_1$ and so on. 
To obtain $\tilde G_{i+1}$ we extend all non-conjugate centralizers of $\tilde G_i$.  
It is clear that $G^{\Z[t]}$ can be obtained as the direct limit of a chain of  extensions of centralizers:
$$G=\tilde G_0\leq_{CE}\tilde G_1\leq _{CE}\ldots \leq _{CE}\tilde G_i\leq _{CE}\ldots .$$ 
At the same time $H$ is the direct limit of a chain of ICE-embeddings of groups from $ICE(G)$, namely
$$G=G_0\leq _{CE}G_1\leq _{CE}\ldots \leq _{CE}G_i\leq _{CE}\ldots ,$$

Since ICE-$age (G^{\Z[t]})=ICE(G)$, for each $i$ there are $j, k,m$ such that
$$\tilde G_i\leq G_j\leq\tilde G_k\leq G_m.$$ Property ICE-N3 implies that
$$\tilde G_i\to _{ICE} G_j\to _{ICE}\tilde G_k\to_{ICE} G_m$$ and $H$ is isomorphic to $G^{\Z[t]}$. 
\end{proof}

We remark that if $G$ is a non-abelian fully residually free group then the class of finitely generated subgroups of $G^{\Z[t]}$ is exactly the class of limit groups. Indeed, since  $F\leq G$ 
we get,  by \cite{KM}, that $G^{\Z[t]}$ contains all limit groups. On the other hand, since the class of limit groups is closed under extensions of centralizers  
and finitely generated subgroups, all finitely generated subgroups of $G^{\Z[t]}$ are limit groups.

If $G$ is a toral relatively hyperbolic group, then, it follows from \cite{KhMyasn:2012}, that the class of finitely generated subgroups of 
$G^{\Z[t]}$ is exactly the class of fully residually $G$ groups.

Theorem \ref{th:Lyndon} explains the model-theoretic nature of Lyndon's completions $G^{\Z[t]}$ for a countable  non-abelian torsion free CSA-group G.

\subsection{Hierarchy of free products and extensions of centralizers}

Denote by $\CH$ the smallest set of groups obtained from the trivial group by finitely many operations of free products with $\mathbb Z$ and extensions of centralizers. We refer to $\CH$ as the hierarchy of centralizer extensions. 
We say that an embedding $\phi: A \to B$ is an $FPCE$-embedding if $B$ can be obtained from $\phi(A)$ by a finite sequence of free products with $\mathbb Z$ and extensions of centralizers.

\begin{lemma} \label{le:AP-H}
The class $\CH$ satisfies JEP, AP, HP and (N3) conditions with respect to $FPCE$-embeddings.
\end{lemma}
\begin{proof}
Observe first that for any groups $A$ and $B$ one has $A(u,t)\ast B \simeq (A\ast B)(u,t)$. Indeed, 
the centralizer of  any non-trivial $u$ in $A$ is equal to the centralizer of $u$ in $A \ast B$. 
Now the claim follows from observing the  defining relations of the groups. 
This shows that the result of applying a finite sequence of operations FP (free product) and CE (centralizer extension) to a group $C$  is the same as applying some FP and after that some CE.

Now let $A$ be $B$ are $FPCE$-extensions of some $C$ in $\CH$. Then by the observation above $A$ is $ICE$-extension of $C\ast F(X)$ for some free group 
$F(X)$ with finite $X$, and $B$ is an $ICE$-extension of $C \ast F(Y)$ for some finite $Y$. 
Consider the $FPCE$-extensions $A_1 = A \ast F(Y)$ and $B_1 = B\ast F(X)$  and put $C_1 = C\ast F(X)\ast F(Y)$. 
Then again by the observation above $A_1, B_1 \in ICE(C_1)$. Applying AP in the class $ICE(C_1)$ one gets the required amalgamation of $A_1$ and $B_1$ with respect  
to $FPCE$-embeddings $C_1 \to A_1$ and $C_1 \to B_1$.

The operation  FP gives JEP in $\CH$.

Conditions HP and (N3) are proved as in Lemma \ref{le:4.2}

\end{proof}

Lemma \ref{le:AP-H} shows that $\CH$ forms an $FPCE$-Fra\"{i}ss\'{e} class. Hence we have the following result.

\begin{theorem} \label{th:Fraisse-H}
 There exists a unique up to isomorphism group $H$ such that
\begin{itemize}
\item [1)] $H$ is countable,
 \item [2)] Finitely generated subgroups of $H$ are exactly limit groups.
 \item [3)] $H$ is the union of an $FPCE$-chain of groups from $\CH$,
 \item [4)] $age_{FPCE}(H) = \CH$,
 \item [5)] $H$ is $FPCE$-homogeneous.
\end{itemize} 
Furthermore, the group $H$ is isomorphic to the Lyndon's completion $\F_{\omega}^{\Z[t]}$ of the free group $\F_{\omega}$ of countable rank. 
\end{theorem}

\section{The free group}

The basis of this section is the following celebrated result.

\begin{theorem}[Kharlampovich-Myasnikov \cite{KM06}, Sela \cite{Sel7}]
The chain $\F_2\subset\F_3\subset\ldots\subset\F_n\subset\ldots$, under the natural embeddings, is an elementary chain. In particular, 
all nonabelian free groups are elementarily equivalent.
\end{theorem}

The {\em theory of the free group} has attracted much attention from both communities, model theorists and group theorists. Apart from its natural models, nonabelian free groups, 
it is satisfied by groups which are not free. As a matter of fact, since the theory of the free group admits uncountably many countable models, the majority of the countable models 
are not free. Despite the complexity of the theory of the free group Kharlampovich-Myasnikov and Sela characterized, group theoretically, its finitely generated models.  

In this section we first give the above mentioned description and then prove that the class of finitely generated elementary free groups is an elementary Fra\"{i}ss\'{e} class. 

The following subsections make heavy use of Bass-Serre theory. We refer the unfamiliar reader to the classical book of Serre \cite{Serre}, or the quick introduction in \cite{Sklinos}. 

\subsection{Hyperbolic towers}

\begin{definition}[Surface type vertices]
Let $G$ be a group acting on a tree $T$ without inversions and $(T^1,T^0,\{\gamma_e\})$ be a Bass-Serre presentation for this action. 
Then a vertex $v\in T^0$ is called a surface type vertex if the following conditions hold:
\begin{itemize}
 \item $\Stab_G(v)=\pi_1(\Sigma)$ for a connected compact surface $\Sigma$ with non-empty boundary such
that either the Euler characteristic of $\Sigma$ is at most -2 or $\Sigma$ is a once punctured torus;
 \item For every edge $e\in T^1$ adjacent to $v$, $\Stab_G(e)$ embeds onto a maximal boundary 
 subgroup of $\pi_1(\Sigma)$, and this induces a one-to-one correspondence between the 
 set of edges (in $T^1$) adjacent to $v$ and the set of boundary components of $\Sigma$.
\end{itemize}

\end{definition}

\begin{definition}[Hyperbolic Floor]
Let $G$ be a group and $H$ be a subgroup of $G$. Then $G$ is a hyperbolic floor over $H$,  
if $G$ acts minimally on a tree $T$ and the action admits a Bass-Serre presentation 
$(T^1, T^0, \{\gamma_e\})$, where the set of vertices of $T^1$ is partitioned in two subsets, $V_S$ and $V_R$, 
such that:
\begin{itemize}
\item each vertex in $V_S$ is a surface type vertex;   
\item the tree $T^1$ is bipartite between $V_S$ and $V_R$; 
\item the subgroup $H$ of $G$ is the free product of the stabilizers of vertices in $V_R$;
\item either there exists a retraction $r:G\to H$ that, for each $v\in V_S$, sends $\Stab_G(v)$ to a non abelian image  
or $H$ is cyclic and there exists a retraction $r': G * \Z \to H * \Z$ that, for each $v\in V_S$, sends $\Stab_G(v)$ to a non abelian image.
\end{itemize}

\end{definition}

 \begin{figure}[ht!]
\centering
\includegraphics[width=.7\textwidth]{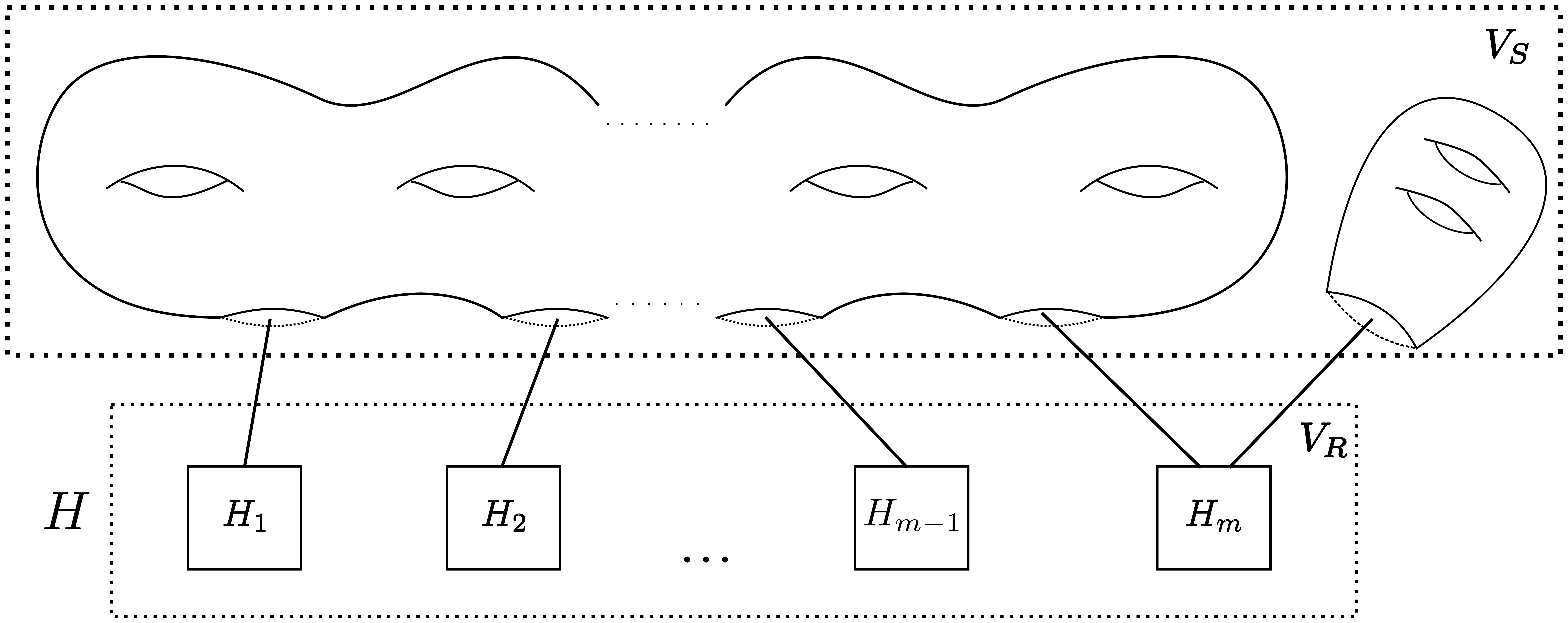}
\caption{A hyperbolic floor.}
\end{figure}

A hyperbolic tower is a sequence of hyperbolic floors and free products with finite rank free groups.

\begin{definition}[Hyperbolic Tower]
A group $G$ is a hyperbolic tower over a subgroup $H$ if there 
exists a sequence $G=G^m>G^{m-1}>\ldots>G^0=H$ such that for each $i$, $0\leq i<m$, 
one of the following holds:
\begin{itemize}
 \item[(i)] the group $G^{i+1}$ has the structure of a hyperbolic floor over $G^i$, in which $H$ is contained in one of the vertex groups that generate $G_i$ in the 
 floor decomposition of $G^{i+1}$ over $G^i$;
 \item[(ii)] the group $G^{i+1}$ is a free product of $G^{i}$ with a finite rank free group.
\end{itemize}
\end{definition}

Hyperbolic towers over a given subgroup amalgamate well in the following sense.

\begin{lemma}\label{AmalgamatedTowers}
Suppose $H_1, H_2$ are hyperbolic towers over $A$. Then the amalgamated free product $H_1*_AH_2$ is a hyperbolic tower over $H_i$ for each $i\leq 2$. 
\end{lemma}

Hyperbolic towers describe finitely generated elementary free groups (or even more generally finitely generated elementary torsion free hyperbolic groups) as well as elementary embeddings amongst torsion free hyperbolic groups. 

\begin{theorem}
A finitely generated group $G$ is elementarily equivalent to a nonabelian free group $F$ if and only if $G$ is a nonabelian hyperbolic tower over $F$. 
\end{theorem}

\begin{theorem}[Kharlampovich-Myasnikov, Sela]\label{TowerElem}
Let $G$ be torsion-free hyperbolic. Let $G$ be a hyperbolic tower over a nonabelian subgroup $H$. Then $H$ is an 
elementary subgroup of $G$.
\end{theorem}

\begin{theorem}[Perin \cite{Perin}]\label{ElemTower}
Let $G$ be a torsion-free hyperbolic group. Let $H$ be an elementary subgroup of $G$. 
Then $G$ is a hyperbolic tower over $H$. 
\end{theorem}

\noindent{\bf Convention:} From now on we restrict our attention to finitely generated elementary free groups. A hyperbolic tower, unless otherwise specified, will always mean a hyperbolic tower over $F$.

We work towards proving.

\begin{theorem}\label{FraisseFree}
The class $\mathcal F_e$  of finitely generated models of the theory of the free group is an elementary Fra\"{i}sse class. 
\end{theorem}

\begin{definition}
Let $H$ be a hyperbolic tower. A subgroup $T$ of $H$ is a hyperbolic subtower if $H$ is a hyperbolic tower over $T$. 
\end{definition}

We remark that a free factor of a hyperbolic tower is a subtower because (see \cite[Definition 7.5 and the paragraph after]{Sel7} and \cite[Corollary 9.2 and Theorems E and F]{KMLifting}):

\begin{theorem}[Kharlampovich-Myasnikov, Sela]\label{FreeProductFreeGroup}
Let $G:=G_1*G_2*\ldots *G_n$ be a model of the theory of the free group. Then for each $i$, either $G_i$ is a model of the theory of the free group or it is infinite cyclic. 
\end{theorem}

We define a partial order on the class of hyperbolic subtowers of a given tower $H$: 
if $T_1, T_2$ are hyperbolic subtowers of $H$ we say that $T_1$ is smaller than $T_2$ if $T_2$ is a hyperbolic tower over $T_1$. It is 
not hard to see, using the descending chain condition for limit groups, that this partial order admits minimal elements. We prove that 
minimal hyperbolic subtowers (relative to any set of parameters) are unique up to isomorphism.

\begin{proposition}\label{MinimalTowers}
Let $H$ be a tower and $\bar{a},\bar{b}$ be tuples from $H$. Suppose $H_1$ (respectively $H_2$) is a minimal hyperbolic subtower of $H$ that contains $\bar{a}$ (respectively $\bar{b}$). Let  $tp^{H_1}(\bar{a})=tp^{H_2}(\bar{b})$. 
Then there exists an isomorphism $h:H_1\rightarrow H_2$ taking $\bar{a}$ to $\bar{b}$. 
\end{proposition}

In order to prove Proposition \ref{MinimalTowers} we use the following results.

\begin{lemma}[\cite{PS12}]\label{retractionHypFloor}
Let $H_1, H_2$ be torsion-free hyperbolic groups. Suppose $H_1$ is freely indecomposable with respect to a nontrivial subgroup $\langle\bar{a}\rangle$ and $tp^{H_1}(\bar{a})=tp^{H_2}(\bar{b})$ for some $\bar{b}\in H_2$. 
Then either there exists an embedding 
$h:H_1\rightarrow H_2$ that sends $\bar{a}$ to $\bar{b}$ or $H_1$ admits the structure of a hyperbolic floor over a subgroup that contains $\langle\bar{a}\rangle$. 
\end{lemma}

The next Lemma can be proved following Sela's ideas on proving the co-Hopf property of freely indecomposable torsion-free hyperbolic groups. Nevertheless, a proof has been given in \cite[Corollary 4.19]{Perin}.

\begin{lemma}[relative co-Hopf property]\label{CoHopf}
Let $H$ be a torsion-free hyperbolic group freely indecomposable with respect to a nontrivial subgroup $A$. Then any injective endomorphism of $H$ that fixes $A$ is surjective.  
\end{lemma}

\noindent{\em Proof of Proposition \ref{MinimalTowers}:} We may assume that $\bar{a}$ is non trivial, as the minimal subtower that contains the trivial element is the trivial subgroup. 

Since $H_1$ is a minimal subtower that contains $\bar{a}$ it must be freely indecomposable with respect to $\langle\bar{a}\rangle$. Therefore, Lemma \ref{retractionHypFloor} applies. Since 
$H_1$ is minimal (relative to $\bar{a}$) it cannot have a hyperbolic floor structure over a subgroup that contains $\bar{a}$, hence there is an injective morphism $h_1:H_1\rightarrow H_2$ 
sending $\bar{a}$ to $\bar{b}$. Similarly, using the minimality of $H_2$, we can find an injective morphism $h_2:H_2\rightarrow H_1$ sending $\bar{b}$ to $\bar{a}$. Their composition 
$h_2\circ h_1$ is an injective endomorphism of $H_1$ that fixes $\bar{a}$, thus, by Lemma \ref{CoHopf}, it is surjective. In particular, the injective morphism $h_1$ is also surjective.\qed
\ \\ \\
We can now prove that the class $\mathcal F_e$  forms an $e$-Fra\"{i}ss\'{e} class.

The $e$-$HP$ condition is immediate: a finitely generated elementary subgroup of an elementary free group is an elementary free group itself. 

The $e$-$JEP$ condition follows rather easily from the characterization of finitely generated models of the theory of the free group, together with Theorems \ref{TowerElem} and \ref{FreeProductFreeGroup}. 

\begin{lemma}\label{FreeProduct}
Suppose $H_1, H_2$ are finitely generated models of the theory of the free group. Then $H_1*H_2$ is a finitely generated model of the theory of the free group. 
Moreover, each $H_i$ for $i\leq 2$ is an elementary subgroup of $H_1*H_2$. 
\end{lemma}

We finally prove the strong $e$-$AP$ condition for the class of finitely generated elementary free groups. 

\begin{lemma}\label{strongeAP}
Let $H_0, H_1, H_2\in \mathcal F_e$.  Let $\bar{a}$ be a tuple from $H_0$ and $f_i:\bar{a}\rightarrow H_i$ for $i= 1, 2$ be partial elementary embeddings. Then there exist 
a  group $D\in \mathcal F_e$ and elementary embeddings $g_i:H_i\rightarrow D$ for $i\leq 2$, such that $g_1\circ f_1(\bar{a})=g_2\circ f_2(\bar{a})$.  
\end{lemma}

\begin{proof}
We may assume that $\bar{a}$ is nontrivial. The hypotheses imply that there are $\bar{b}\in H_1$ and $\bar{c}\in H_2$  such that $tp^{H_1}(\bar{b})=tp^{H_2}(\bar{c})=tp^{H_0}(\bar{a})$. 
We take cases according to whether the smallest subtower $T_1$ of $H_1$ that contains $\bar{a}$ is abelian or not. 

Suppose $T_1$ is abelian, hence cyclic. We prove that $T_2$ the smallest subtower that contains $\bar{c}$ is cyclic as well. Suppose, for the sake of contradiction, not, then  Theorem \ref{TowerElem} implies 
$tp^{H_1}(\bar{b})=tp^{T_2}(\bar{c})$. Moreover, $tp^{H_1*\Z}(\bar{b})=tp^{T_1*\Z}(\bar{b})=tp^{T_2}(\bar{c})$ and since $T_2$ is minimal it must embed into $T_1*\Z$, by an embedding that sends $\bar{c}$ to $\bar{b}$. 
Since, $T_2$ is freely indecomposable with respect to $\bar{c}$, it must embed in $T_1$, hence it is cyclic, a contradiction. Next, since both $T_1$, $T_2$ are cyclic and $tp^{T_1*\Z}(\bar{b})=tp^{T_2*\Z}(\bar{c})$, we can find an isomorphism 
$h:T_1\rightarrow T_2$ that sends $\bar{b}$ to $\bar{c}$.

Suppose $T_1$ (and consequently $T_2$) is nonabelian. Hence, Theorem \ref{TowerElem} implies $tp^{T_1}(\bar{b})=tp^{T_2}(\bar{c})$. By Proposition \ref{MinimalTowers}, applied to $H=H_1*H_2$, 
we obtain an isomorphism $h:T_1\rightarrow T_2$ sending $\bar{b}$ to $\bar{c}$. 

Finally, in both cases, by Lemma \ref{AmalgamatedTowers}, the amalgamated free product $H_1*_{T_1}H_2$ (using $h$ for the amalgamation) is a finitely generated elementary free group. Hence is the group $D$ we wanted. 
\end{proof}

We thus obtain a countable elementary free group with interesting properties. 

\begin{theorem}\label{ElementaryFreeFraisse}
There exists a countable group $G_e$ with the following properties: 
\begin{itemize} 
 \item the $e$-$age$ of $G_e$ is the class of finitely generated elementary free groups;
 \item the group $G_e$ is homogeneous; 
 \item the group $G_e$ is the union of an elementary chain of finitely generated elementary groups.
\end{itemize}
Moreover, any countable group with the above properties is isomorphic to $G_e$.
\end{theorem}

Finally as a by-product of the already proved theorems in this section we obtain. 

\begin{theorem}
The class of finite rank nonabelian free groups is a strong elementary Fra\"{i}ss\'{e} class. In particular $\F_{\omega}$ is homogeneous.  
\end{theorem}

This last result also follows from the facts that finite rank nonabelian free groups are homogeneous (see \cite{PS12}, \cite{OH11}) and $\F_{\omega}$ is the union of their elementary chain.

\section{Further remarks and open questions}

\begin{lemma}
 There are countable elementary free groups which are not limits of $\forall$-chains of finitely generated elementary free groups.
\end{lemma}

\begin{proof}
The result follows by the following observation in \cite{SklinosThesis}. Let $P$ be the set of primes. For any $X\subseteq P$, the following set 
of formulas $\Gamma_X:=\{ \exists y(x=y^p) \ | \ p\in X\}\cup \{ \forall y (x\neq y^q) \ | \ q\in P\setminus X\}$ is consistent. Moreover, 
for any two $X, Y \subseteq P$ with $X\neq Y$ the union $\Gamma_X\cup\Gamma_Y$ is inconsistent. 

\end{proof}

\noindent {\bf Questions:} 
\begin{itemize}
 \item Do finitely generated elementary free groups (or finite rank nonabelian free groups) form a (strong) $\forall$-Fra\"{i}ss\'{e} class ?
 \item Does there exist a countable homogeneous elementary free group whose elementary age is the class of all finitely generated elementary free groups and 
 is not isomorphic to the strong $e$-Fra\"{i}ss\'{e} limit of this class ?
 \item Do the nonabelian limit groups form a strong $\forall$-Fra\"{i}ss\'{e} class ? 
 \item Are all countable elementary free groups obtained as the union of a chain of finitely generated elementary free groups ?
\end{itemize}

We would like to thank the referee for the important comments which  helped improving the quality of the paper.


\begin{thebibliography}{99}

\bibitem{B} B. Baumslag. {\em Residually free groups}. Proc. Lond. Math. Soc., {\bf 17} (1967), pp. 402--418.
\bibitem{DE} D. Evans. \href{https://www.him.uni-bonn.de/uploads/media/Homogeneous\_structures\_omega\_categoricity\_and\_amalgamation\_constructions.pdf}{Notes on homogeneous structures, omega-categoricity and amalgamation constructions}, 
Hausdorff Institute for Mathematics, Trimester program in Universality and Homogeneity, Bonn, 2013.   
\bibitem{RF} R. Fra\"{i}ss\'{e}. {\it Sur l'extension au relations de quelques propri\'{e}t\'{e}s des ordres}. Annales scientifiques de l'\'{E}.N.S., $3^e$ s\'{e}rie, tome 71 (1954), no 4, pp. 363--388. 
\bibitem{KM}  O. Kharlampovich, A. Myasnikov. Irreducible affine varieties over a free group. II: Systems in row-echelon form and description of residually free groups. J. Algebra, {\bf 200} (1998), pp. 517--570.
\bibitem{KM06} O. Kharlampovich, A.Myasnikov. \emph{{Elementary theory of free nonabelian groups}}. J. Algebra, {\bf 302} (2006), pp. 451--552.
\bibitem {KhMyasn:2012} O. Kharlampovich, A. Myasnikov. {\it Limits of relatively hyperbolic groups and Lyndon's completions}. Journal of the European Mathematical Society, {\bf 14} (2012), no. 3, pp. 659--680.
\bibitem{KMLifting} O. Kharlampovich, A. Myasnikov. {\it Algebraic geometry over free groups: lifting solutions into generic points}. Contemporary Mathematics, {\bf 378} (2005), pp. 231--317.
\bibitem {KMeq} O. Kharlampovich, A. Myasnikov, {\em Equations and fully residually free groups}.  Combinatorial and Geometric Group Theory, Trends in Mathematics, pp. 203--242, 2010.
\bibitem{LySch} R. C. Lyndon, P. E. Schupp. {\em Combinatorial Group Theory}. Springer-Verlag Berlin Heidelberg. 1977. 
\bibitem{Magnus} W. Magnus, A. Karrass, D. Solitar. {\em Combinatorial Group Theory: Presentations of Groups in Terms of Generators and Relations}. Dover Books on Mathematics. 1976.  
\bibitem{MR} A. G. Myasnikov,  V. N. Remeslennikov. {\em Exponential groups. II. Extensions of centralizers and tensor completion of CSA-groups}. International J. Algebra  Comput., {\bf 6} (1996), pp. 687--711.

\bibitem{OH11} A. Ould Houcine. {\em Homogeneity and prime models in torsion-free hyperbolic groups}. Confluentes Mathematici, \textbf{3} (2011), no. 1, pp. 121--155.
\bibitem{Perin} C. Perin. {\em Elementary embeddings in torsion-free hyperbolic groups}. Annales de l'\'{E}NS, $4^e$ s\'{e}rie, tome 44 (2011), no. 4, pp. 631--681.

\bibitem{PS12} C. Perin and R. Sklinos. {\em Homogeneity in the Free Group}. Duke Math. J., \textbf{161} (2012), no. 13, pp. 2635--2668.
\bibitem{Sela1} Z. Sela. {\em Diophantine geometry over groups I: Makanin-Razborov diagrams.} Publications Math\'{e}matiques de l'IH\'{E}S, {\bf 93} (2001), pp. 31--105.
\bibitem{Sel7} Z. Sela. \emph{{Diophantine geometry over groups. {VII}. {T}he elementary theory of a hyperbolic group}}. Proc. Lond. Math. Soc. (3), {\bf 99} (2009), pp. 217--273.
\bibitem{Serre} J.-P. Serre. {\em Trees} (translated by J. Stilwell). Springer-Verlag Berlin Heidelberg. 1980. 
\bibitem{Sklinos} R. Sklinos. {\em The free group does not have the finite cover property}. Israel J. Math., {\bf 227} (2018), no 2, pp. 563--595, 2018. 
\bibitem{SklinosThesis} R. Sklinos. {\em Some model theory of the free group}. PhD Thesis. University of Leeds, 2011. 


\end{thebibliography}
\end{document}